\documentclass[11pt]{article}
\usepackage{pifont}

\usepackage{amsfonts}
\usepackage{latexsym}
\usepackage{amsmath}
\usepackage{amssymb}
\usepackage{color}

\newtheorem{theorem}{Theorem}[section]

\newtheorem{proposition}{Proposition}[section]

\newenvironment{proof}[1][Proof]{\noindent \textbf{#1.} }{\ \ \  $\Box$}

\newtheorem{lemma}{Lemma}[section]
\newtheorem{definition}{Definition}[section]
\newtheorem{remark}{Remark}[section]

\title{A maximum principle for forward-backward stochastic Volterra integral
 equations and applications in finance\thanks{This work is supported by National Natural
Science Foundation of China Grant 10771122, Natural Science
Foundation of Shandong Province of China Grant Y2006A08 and National
Basic Research Program of China (973 Program, No. 2007CB814900).}}

\date{April 12 2010}

\author{Tianxiao Wang and Yufeng Shi \thanks{E-mail: xiaotian2008001@gmail.com, yfshi@sdu.edu.cn}\\ \small{School of
Mathematics, Shandong University, Jinan 250100, China}}

\begin{document}

\maketitle

\begin{abstract}
This paper formulates and studies a stochastic maximum principle for
forward-backward stochastic Volterra integral equations (FBSVIEs in
short), while the control area is assumed to be convex. Then a
linear quadratic (LQ in short) problem for backward stochastic
Volterra integral equations (BSVIEs in short) is present to
illustrate the aforementioned optimal control problem. Motivated by
the technical skills in solving above problem, a more convenient and
briefer method for the unique solvability of M-solution for BSVIEs
is proposed. At last, we will investigate a risk minimization
problem by means of the maximum principle for FBSVIEs. Closed-form
optimal portfolio is obtained in some special cases.

\par  $\textit{Keywords:}$  Forward-backward stochastic Volterra integral
equations,
   Adapted M-solution, Optimal control, Stochastic maximum
   principle, Backward linear quadratic, Risk minimization problem
\end{abstract}



\section{Introduction}\label{sec:intro}
Throughout this paper we assume that all uncertainties come from a
common complete probability space $(\Omega ,\mathcal {F},P)$ on
which is defined a $d$-dimensional Wiener process $(W_t)_{t\in
[0,T]}$. The main objective of this paper is to study the optimal
control problem for the following forward-backward stochastic
Volterra integral equation (FBSVIE in short)
\begin{equation}
\left\{
\begin{array}{c}
X(t)=\varphi
(t)+\displaystyle\int_0^tb(t,s,X(s))ds+\displaystyle\int_0^t\sigma
(t,s,X(s))dW(s), \\
Y(t)=\psi
(t)+\displaystyle\int_t^Tg(t,s,X(s),Y(s),Z(s,t))ds-\displaystyle\int_t^TZ(t,s)dW(s),
\end{array}
\right.
\end{equation}
which for instance generalize the optimal control problems for
stochastic Volterra integral equations in \cite{Y2}.

The notion of M-solution for backward stochastic Volterra integral
equation (BSVIE in short) of the form
\begin{eqnarray}
Y(t)=\psi(t)+\int_t^Tg(t,s,Y(s),Z(s,t),Z(t,s))ds-\int_t^TZ(t,s)dW(s)
\end{eqnarray}
with $t\in[0,T]$, was introduced by Yong in \cite{Y2}, which plays
an important role in optimal control problem for stochastic Volterra
integral equations (SVIEs for short). We refer the author to Lin
\cite{L}, Yong (\cite{Y1}, \cite{Y3}), Wang and Zhang (\cite{WZ})
for a study of the wellposedness  of BSVIEs in finite space, while
Anh and Yong \cite{AY}, Ren \cite{R} in infinite space counterpart.

One main feature of equation (2) lies in the dependence of the
generator $g$ on $Z(s,t)$, and hence it is quite different from,
more precisely, a natural generalization of the one in \cite{L} and
\cite{WZ}. Of course, the appearance of such term in $g$ is not just
means an extension from mathematical point of view, but also be of
great importance in applications, (see Proposition 3.5 in \cite{Y3}
and Theorem 5.1 in \cite{Y2}.) It is interesting to realize that, so
far as we know, it is just the term $Z(s,t)$ rather than $Z(t,s)$ in
the generator $g$ that plays a key role in both optimal control
problem in \cite{Y2} and dynamic risk measure in \cite{Y3}.

Optimal control of forward stochastic differential systems is a
classical problem. When we consider the Pontryagin maximum principle
for optimal controls of stochastic differential equations, the
adjoint equation for variational state equation actually is a linear
backward stochastic differential equation (BSDE for short). The
wellposedness for nonlinear BSDEs was firstly studied by Pardoux and
Peng \cite{PP}. Readers interested in an in-depth analysis of BSDEs
can see the books of Ma and Yong \cite{MY}, Yong and Zhou \cite{YZ}
and the survey paper of EI Karoui, Peng and Quenez \cite{EPQ}. As to
the optimal control for stochastic differential equation, we refer
the reader to, for example, Peng \cite{P1} for the general case of
control domain being non-convex, and Yong and Zhou \cite{YZ} for
systematical analysis. On the other hand, optimal control for
deterministic Volterra integral equation, particularly linear
quadratic problem, was firstly studied by Vinokurov \cite{V}. From
then on some other extensions were developed, see, for example,
\cite{B1}, \cite{B2}, \cite{PY}, \cite{You} and the references cited
therein. As to the stochastic version, Yong (\cite{Y1} and
\cite{Y2}) presented a maximum principle for SVIEs by means of
BSVIEs, while the control is assumed to be convex. We also would
like to mention the work of {\O}ksendal and Zhang \cite{OZ} in
partial information setting without the help of BSVIEs. Along this,
we will investigate the FBSVIEs case in this paper. To the best of
our knowledge, so far little is known about maximum principle for
FBSVIEs, and one aim of this manuscript is to close the gap.

The scheme is designed around the three steps for FBSDEs in Peng
\cite{P2}, namely listing out the variational equation, obtaining
the variational inequality and utilizing some key mathematical tools
to finish the procedure. As we know, within the context of
stochastic differential systems, It\^{o} formula has received most
attention largely due to its ad hoc role in many complicated
calculations and proofs. For example one usually makes use of
It\^{o} formula in obtaining the convergence property for
$\widetilde{X}_{\rho}$ and $\widetilde{Y}_{\rho}$ (defined blew) for
differential systems, see Lemma 4.1 in \cite{P2}. In fact, one key
tool in deriving maximum principle in Peng \cite{P2} is just It\^{o}
formula too. Unfortunately, this efficient tool is failure in the
Volterra integral systems and some related well properties are
absent in this case.

In this paper, new approaches are proposed to handle with the
difficulties encountered in above procedure. On the one hand, we
will make use of the dual principle, established by Yong in
\cite{Y2}, for linear stochastic integral equation and its adjoint
equation. Consequently we have to tackle four equations which
perhaps means more mathematical expressions and notations involved
after introducing another two more adjoint equations for FBSVIEs. As
a result, it is our hope to choose appropriate form of adjoint
equations so as to make the procedure as brief as possible.
Fortunately, such adjoint equations really exist, see (10) and (22).
On the other hand, we introduce a new equivalent norm for elements
in $\mathcal{H}^2[0,T]$, see (6), and use some common calculations
and tricks employed in the conventional BSDEs case, thereby obtain
some convergence results, which play a chief role in deducing the
variational inequality. Notice that It\^{o} formula does not appear
in the above two aspects.

Motivated by the new norm aforementioned, in the following we will
provide a new method for the unique solvability of M-solution, which
seems more convenient than the one in \cite{Y2}. By the four steps
in Theorem 3.7 in \cite{Y2} we can see the process of constructing
the M-solution clearly. From mathematical view, however, the whole
proof is too complicated and uneasy to understand, which prompts us
to seek an alternative one. We will carry out this course in detail
in Section 3.

A class of continuous time dynamic convex and coherent risk
measures, perhaps allowing time-inconsistent preference unlike the
conventional case, were introduced by Yong in \cite{Y3} via BSVIEs
of the form
\begin{eqnarray}
Y(t)=-\psi(t)+\int_t^Tg(t,s,Y(s),Z(s,t))ds-\int_t^TZ(t,s)dW(s),
\end{eqnarray}
with $t\in[0,T].$ In the classical case, the terminal condition is
usually a bounded random variable, representing the financial
position at time $T.$ However, in the situation under our
consideration, we prefer to choose a process $\psi$, representing
the total wealth of certain portfolio process at time $t$ which
might be a combination of certain contingent claims, positions of
stocks, mutual funds and bonds. Usually the $\psi$ could be
$\mathcal{B}[0,T]\otimes\mathcal{F}_{T}$-measurable, see the example
on p. 13 in \cite{Y3}. If we define a map $\varrho$ from
$L^2_{\mathcal{F}_{T}}[0,T]$ to $L^2_{\Bbb{F}}[0,T]$ by
$\varrho(t;\psi(\cdot))=Y(t)$, with $Y$ being the M-solution of
BSVIE (3), given certain assumptions on $g$, it is shown in
\cite{Y3} that $\varrho$ could be a dynamic convex/coherent risk
measure. The question is how to look for a appropriate portfolio
that minimizes the risk of the wealth process $\psi$ by means of the
representation above in finance, i.e., to seek an optimal solution
for the so-called risk minimization problem, see Mataramvura and
{\O}ksendal \cite{MO}, {\O}ksendal and Sulem \cite{OS} for more
information on the above financial problem. We conclude this paper
by giving a study of this problem with the help of the maximum
principle. In some cases, the closed form of optimal portfolio is
derived.

The remainder of paper is organized as follows. In Section 2, we
give some preliminary results and notations which are needed in the
following sections. A new method for the solvability of M-solution
is presented in Section 3. We give the stochastic maximum principle
for FBSVIEs (1) as well as a backward linear quadratic problem in
Section 4. At last, we investigate a risk minimization problem by
means of maximum principle in the previous. Some explicit form
solutions are derived.

\section{Preliminaries}

   In this section, we will make some preliminaries. Let us specify some notation in this paper.
   For any $R,S\in[0,T],$ we denote $\Delta ^c[0,T]=\{(t,s)\in[0,T]^{2};
   t\leq s\},$ $\Delta [0,T]=\{(t,s)\in[0,T]^{2}; t>s\}.$
In what follows some spaces will be frequently used. Let
$L_{\mathcal {F}_{T}}^p[0,T]$ be the set of the
$\mathcal{B}([0,T])\otimes\mathcal{F}_{T}$ processes $X:[0,T]\times
\Omega \rightarrow R^m$ such that $ E\int_0^T|X(t)|^pdt<\infty. $
$L_{\mathcal {F}}^p[0,T]$ is the set of all adapted processes
$X:[0,T]\times \Omega \rightarrow R^m$ such that $
E\int_0^T|X(s)|^pds<\infty . $
$L^p(0,T;L_{\mathcal {F}}^2[0,T])$ is the set of all processes $%
Z:[0,T]^2\times \Omega \rightarrow R^{m\times d}$ such that for almost all $%
t\in [0,T],$ $Z(t,\cdot )$ is $\mathcal {F}$-progressively
measurable satisfying $
E\int_0^T\left(\int_0^T|Z(t,s)|^2ds\right)^{\frac p 2}dt<\infty . $
For notational clarity, we denote $ \mathcal{H}^p[0,T]=L_{\mathcal
{F}}^p[0,T]\times L^p(0,T;L_{\mathcal {F}}^2[0,T]).$
Next we shall cite the definition of M-solution introduced in
\cite{Y2}.
\begin{definition}
A pair of $(Y(\cdot ),Z(\cdot ,\cdot ))\in \mathcal{H}%
^p[0,T]$ is called an adapted $M$-solution of BSVIE (2) on $[0,T]$
if (2) holds in the usual It\^o's sense for almost all $t\in [0,T]$
and, in addition, $ Y(t)=EY(t)+\int_0^tZ(t,s)dW(s)$ with $ t\in
[0,T]. $
\end{definition}
The next two definitions are introduced by Yong in \cite{Y3}.
\begin{definition}
A mapping $\rho :L_{\mathcal{F}_T}^2[0,T]\rightarrow
L_{\mathbb{F}}^2[0,T]$ is called a dynamic risk measure if the
following hold:

1) (Past independence) For any $\Psi (\cdot ),$ $\overline{\Psi }(\cdot )\in L_{\mathcal{F}%
_T}^2[0,T],$ if $\Psi (s)=\overline{\Psi }(s),$ a.s. $\omega \in \Omega ,$ $%
s\in [t,T],$ for some $t\in [0,T),$ then $\rho (t;\Psi (\cdot ))=\rho (t;%
\overline{\Psi }(\cdot )),$ a.s. $\omega \in \Omega .$

2) (Monotonicity) For any $\Psi (\cdot ),$ $\overline{\Psi }(\cdot )\in L_{\mathcal{F}%
_T}^2[0,T],$ if $\Psi (s)\leq \overline{\Psi }(s),$ a.s. $\omega \in
\Omega , $ $s\in [t,T],$ for some $t\in [0,T),$ then $\rho (s;\Psi
(\cdot ))\geq \rho (s;\overline{\Psi }(\cdot )),$ a.s. $\omega \in
\Omega, \text{ } s\in[t,T].$
\end{definition}
\begin{definition}
A dynamic risk measure $\rho :L_{\mathcal{F}_T}^2[0,T]\rightarrow L_{\mathbb{%
F}}^2[0,T]$ is called a coherent risk measure if the following hold:

1) There exists a deterministic integrable function $r(\cdot )$ such
that for
any $\Psi (\cdot )\in L_{\mathcal{F}_T}^2[0,T],$%
\begin{equation*}
\rho (t;\Psi (\cdot )+c)=\rho (t;\Psi (\cdot
))-ce^{\int_t^Tr(s)ds},\text{ a.s. }\omega \in \Omega ,\text{ }t\in
[0,T].
\end{equation*}

2) For $\Psi (\cdot )\in L_{\mathcal{F}_T}^2[0,T]$ and $\lambda >0,$
$\rho (t;\lambda \Psi (\cdot ))=\lambda \rho (t;\Psi (\cdot ))$ a.s.
$\omega \in \Omega ,$ $t\in [0,T].$

3) For any $\Psi (\cdot ),$ $\overline{\Psi }(\cdot )\in L_{\mathcal{F}%
_T}^2[0,T],$
\begin{equation*}
\rho (t;\Psi (\cdot )+\overline{\Psi }(\cdot ))\leq \rho (t;\Psi
(\cdot
))+\rho (t;\overline{\Psi }(\cdot )),\text{ a.s. }\omega \in \Omega ,\text{ }%
t\in [0,T].
\end{equation*}
\end{definition}

Some necessary specifications on the generator $g$ for BSVIE (2) are
given by:

(H1) Let $g:\Delta ^c\times R^m\times R^{m\times d}\times R^{m\times
d}\times \Omega \rightarrow R^m$ be $\mathcal{B}(\Delta ^c\times
R^m\times R^{m\times
d}\times R^{m\times d})\otimes \mathcal{F}_T$-measurable such that $%
s\rightarrow g(t,s,y,z,\zeta )$ is $\mathcal {F}$-progressively
measurable for all $(t,y,z,\zeta )\in [0,T]\times R^m\times
R^{m\times d}\times R^{m\times d}$, and $\forall y,$
$\overline{y}\in R^m,$ $z,$ $\overline{z},$ $\zeta ,$
$\overline{\zeta }\in R^{m\times d},$
\begin{eqnarray*}
&&|g(t,s,y,z,\zeta )-g(t,s,\overline{y},\overline{z},\overline{\zeta })| \\
&\leq &L_1(t,s)|y-\overline{y}|+L_2(t,s)|z-\overline{z}|+L_3(t,s)|\zeta -%
\overline{\zeta }|,
\end{eqnarray*}
where $(t,s)\in \Delta ^c,$ $L_i(t,s)$ $(i=1,2,3)$ is deterministic
non-negative functions. Furthermore $
E\int_0^T\left(\int_t^T|g_0(t,s)|ds\right)^pdt<\infty ,$ where
$g_0(t,s)=g(t,s,0,0,0).$

\section{A new method for unique solvability of M-solution}

 In this section, a new scheme is proposed and analyzed to simplify the unique
 solvability of M-solution
in Yong \cite{Y2}. The proof in \cite{Y2} gives us a detailed
procedure to comprehend how to construct M-solutions, however, from
a mathematical point of view, it is rather tedious and
sophisticated, and it should be of interest to develop a new brief
approach for it.

Inspired by the following equivalent norm for the elements of
$\mathcal{H}^2[0,T]$ in \cite{WS},
\[
\left\| (y(\cdot ),z(\cdot ,\cdot ))\right\|
_{\mathcal{H}^2[0,T]}=\left[ E\int_0^Te^{\beta
t}|y(t)|^2dt+E\int_0^T\int_0^Te^{\beta s}|z(t,s)|^2dsdt\right]
^{\frac 12},
\]
with $\beta $ being a positive constant, we can propose a new one,
see (6), and thus achieve the goal of giving a convenient and brief
proof. In addition, compared with the proof in \cite{WS}, it seems
that the proof here is still simpler. Furthermore, we can also
handle with the general case for $p\in(1,2]$ with this approach.


Before doing this, some preparations are required. Consider the
following simple BSVIE,
\begin{equation}
Y(t)=\psi (t)+\int_t^Th(t,s,Z(t,s))ds-\int_t^TZ(t,s)dW(s).
\end{equation}
(H2) $h$ has the similar assumptions with $g$ in (H1). Furthermore,
$L_2(t,s)$ satisfies the condition, $\sup\limits_{t\in
[0,T]}\int_t^TL_2(t,s)^{2+\epsilon }ds<\infty ,$ with some constant
$\varepsilon >0$.

The proof of the next proposition can be found in \cite{Y2}.
\begin{proposition}
Let (H2) hold, then for any $\psi (\cdot )\in L_{\mathcal{F}%
_T}^p[0,T]$, (4) admits a unique adapted M-solution $(Y(\cdot
),Z(\cdot ,\cdot ))\in \mathcal{H}^p[0,T]$.
If $\overline{h}$ also satisfies (H2), $\overline{\psi }(\cdot )\in L_{%
\mathcal{F}_{T}}^{p}[0,T],$ and $(\overline{Y}(\cdot
),\overline{Z}(\cdot ,\cdot ))\in \mathcal{H}^{p}[0,T]$ is the
unique adapted M-solution of BSVIE
(4) with $(h,\psi )$ replaced by $(\overline{h},\overline{\psi }),$ then $%
\forall t\in [0,T],$
\begin{eqnarray}
&&\ E\left\{ |Y(t)-\overline{Y}(t)|^{p}+\left(\int_{t}^{T}|Z(t,s)-\overline{Z}%
(t,s)|^{2}ds\right)^{\frac p 2}\right\}  \nonumber \\
\  &\leq &CE\left[ |\Psi (t)-\overline{\Psi }(t)|^{p}+\left(
\int_{t}^{T}|h(t,s,Z(t,s))-\overline{h}(t,s,Z(t,s))|ds\right)
^{p}\right] .
\end{eqnarray}%
Hereafter $C$ is a generic positive constant which may be different
from line to line.
\end{proposition}
We move on to give the main result of this section.
\begin{theorem}
Let (H1) hold, assume that
\begin{eqnarray*}
\sup\limits_{t\in
[0,T]}\int_t^TL_1^q(t,s)ds<\infty,\quad\sup\limits_{t\in
[0,T]}\int_t^TL_2^{2+\epsilon}(t,s)ds<\infty,\quad \sup\limits_{t\in
[0,T]}\int_t^TL_3^{q'}(t,s)ds<\infty,
\end{eqnarray*}
where $\frac 1p+\frac 1q=1$, $%
p\in (1,2],$ $\frac {1}{ p'}+\frac {1}{q'}=1,$ $1<p'<p.$ Then for
any $\psi (\cdot )\in L_{\mathcal{F}_T}^p[0,T]$, BSVIE (1) admits a
unique adapted M-solution in $\mathcal{H}^p[0,T].$
\end{theorem}
\begin{proof}First let $\mathcal{M}^p[0,T]$ be the space
of all $(y(\cdot ),z(\cdot ,\cdot ))\in \mathcal{H}^p[0,T]$ such
that $ y(t)=Ey(t)+\int_0^tz(t,s)dW(s)$ with $t\in [0,T].$ It is a
matter of direct calculation to show that $\mathcal{M}^p[0,T]$ is a
closed subspace of $\mathcal{H}^p[0,T]$ via the following two
martingale moment inequalities in \cite{KS},
\[
E\int_0^T\left| \int_0^tz(t,s)dW(s)\right| ^pdt\leq
C_pE\int_0^T\left( \int_0^t|z(t,s)|^2ds\right) ^{\frac p2}dt,\text{
if }p>0,
\]
and
\[
E\int_0^T\left( \int_0^t|z(t,s)|^2ds\right) ^{\frac p2}dt\leq
C_pE\int_0^T\left| \int_0^tz(t,s)dW(s)\right| ^pdt,\text{ if }p>1,
\]
where $C_p$ is a constant depending on $p.$ A new equivalent norm
for the element in $\mathcal{M}^p[0,T]\ $ of the form
\begin{eqnarray}
\left\| (y(\cdot ),z(\cdot ,\cdot ))\right\|
_{\mathcal{M}^p[0,T]}=\left[ E\int_0^Te^{\beta
t}|y(t)|^pdt+E\int_0^Te^{\beta t}\left( \int_0^T|z(t,s)|^2ds\right)
^{\frac p2}dt\right] ^{\frac 1p},
\end{eqnarray}
will be in force in the following part. Consider,
\begin{eqnarray}
Y(t)=\psi
(t)+\int_t^Tg(t,s,y(s),Z(t,s),z(s,t))ds-\int_t^TZ(t,s)dW(s),
\end{eqnarray}
 with $t\in[0,T],$ $\psi (\cdot )\in L_{\mathcal{F}_T}^p[0,T]$ and
$(y(\cdot ),z(\cdot ,\cdot ))\in \mathcal{M}^p[0,T]$. Following
Proposition 3.1 we get that (7) admits a unique adapted M-solution
$(Y(\cdot ),Z(\cdot ,\cdot ))$, and then define a map $\Theta
:\mathcal{M}^p[0,T]\rightarrow \mathcal{M}^p[0,T]$ by
\[
\Theta (y(\cdot ),z(\cdot ,\cdot ))=(Y(\cdot ),Z(\cdot ,\cdot
)),\quad \forall (y(\cdot ),z(\cdot ,\cdot ))\in \mathcal{M}^p[0,T].
\]
Let $(\overline{y}(\cdot ),\overline{z}(\cdot ,\cdot ))\in \mathcal{M}%
^p[0,T] $ and $\Theta (\overline{y}(\cdot ),\overline{z}(\cdot ,\cdot ))=(%
\overline{Y}(\cdot ),\overline{Z}(\cdot ,\cdot )),$ thus it follows
from inequality (5) that,
\begin{eqnarray*}
&&E\int_0^Te^{\beta t}|Y(t)-\overline{Y}(t)|^pdt+E\int_0^Te^{\beta
t}\left(
\int_t^T|Z(t,s)-\overline{Z}(t,s)|^2ds\right) ^{\frac p2}dt \\
&\leq &CE\int_0^Te^{\beta t}\left\{ \int_t^T|g(t,s,y(s),Z(t,s),z(s,t))-g(t,s,%
\overline{y}(s),Z(t,s),\overline{z}(s,t))|ds\right\} ^pdt \\
&\leq &CE\int_0^Te^{\beta t}\left\{ \int_t^TL_1(t,s)|y(s)-\overline{y}%
(s)|ds\right\} ^pdt \\
&&+CE\int_0^Te^{\beta t}\left\{ \int_t^TL_3(t,s)|z(s,t)-\overline{z}%
(s,t)|ds\right\} ^pdt \\
&\leq &CE\int_0^Te^{\beta t}\left( \sup_{t\in
[0,T]}\int_t^TL_1^q(t,s)ds\right)^{\frac p q} \int_t^T|y(s)-\overline{y}(s)|^pdsdt \\
&&+CE\int_0^Te^{\beta t}\left( \sup_{t\in
[0,T]}\int_t^TL_3^{q'}(t,s)ds\right)^{\frac {p} {q'}}
\left(\int_t^T|z(s,t)-\overline{z}(s,t)|^{p'}ds\right)^{\frac {p} {p'}}dt \\
&\leq &CE\int_0^T|y(s)-\overline{y}(s)|^pds\int_0^se^{\beta
t}dt+C\left[\frac1 \beta\right]^{\frac {p-p'} {p'}} E\int_0^Tds\int_t^Te^{\beta s}|z(s,t)-\overline{z}(s,t)|^pdt \\
&\leq &\frac C\beta E\int_0^Te^{\beta s}|y(s)-\overline{y}(s)|^pds
+C\left[\frac 1 \beta\right]^{\frac {p-p'}{p'}} E\int_0^Te^{\beta t}dt\int_0^t|z(t,s)-\overline{z}(t,s)|^pds \\
&\leq &\frac C\beta E\int_0^Te^{\beta s}|y(s)-\overline{y}(s)|^pds,
\end{eqnarray*}
where $1<p'<p,$ $\frac {1}{p'}+\frac {1 }{q'}=1.$ Notice that here
we use the following two relations, that are, for any $p\in(1,2],$
$1<p'<p,$ $r>0,$
\begin{eqnarray}
&&\left[ \int_t^T|z(s,t)-\overline{z}(s,t)|^{p'}ds\right] ^{\frac
{p}
{p'}} \nonumber \\
&\leq& \left[ \int_t^Te^{-rs\frac p{p-p'}}ds\right] ^{\frac{p-p'}%
{p'}}\int_t^Te^{rs\frac {p} {p'}}|z(s,t)-\overline{z}(s,t)|^pds \nonumber \\
&\leq &\left[ \frac 1 r\right] ^{\frac{p-p'}{p'}}\left[ \frac{p-p'}p\right] ^{%
\frac{p-p'}{p'}}e^{-rt\frac {p}{p'}}\int_t^Te^{rs\frac {p}
{p'}}|z(s,t)-\overline{z}(s,t)|^pds.
\end{eqnarray}
and $E\int_0^t|z(t,s)-\overline{z}(t,s)|^pds \leq
CE|y(t)-\overline{y}(t)|^p$ which is a direct consequence of
martingale moment inequality and H\"{o}lder inequality. Then we can
choose a $\beta ,$ so that the map $\Theta $ is a contraction, and
the result holds.
\end{proof}
%

\section{A maximum principle for FBSVIE}

   In this section, we give a stochastic maximum principle for
forward-backward stochastic Volterra integral equations by assuming
the control domain being convex and $p=2,$ thereby generalizing for
instance the case in \cite{Y2}. As compared with the differential
case, it is by no means clear that the method there can be extended
to such setting. As we have claimed in the previous, there are some
technical obstacles for us to overcome due to the absence of It\^{o}
formula here, in other words, we should adopt some other effective
mathematical skills to circumvent the difficulties caused by it.
Without loss of generality, we assume that $m=d=1.$

\subsection{Setting the problem}

We denote by $U$ a nonempty convex subset of $R,$ and set
\[
\mathcal{U=\{}v(\cdot )\in L_{\mathcal {F}}^2[0,T];  v(t)\in U, a.s.
\quad t\in[0,T], a.e.\}.
\]
An element of $\mathcal{U}$ is called an admissible control. For any
admissible control $v(\cdot )\in \mathcal{U},$ let us consider the
following forward-backward stochastic Volterra integral equation,
i.e.,
\begin{equation}
\left\{
\begin{array}{c}
X(t)=\varphi
(t)+\displaystyle\int_0^tb(t,s,X(s),v(s))ds+\displaystyle\int_0^t\sigma
(t,s,X(s),v(s))dW(s), \\
Y(t)=\psi
(t)+\displaystyle\int_t^Tg(t,s,X(s),Y(s),Z(s,t),v(s))ds-\displaystyle\int_t^TZ(t,s)dW(s),
\end{array}
\right.
\end{equation}
associated with the cost functional by
\[
J(v(\cdot ))=E\left[
\int_0^Tl(s,X(s),Y(s),v(s))ds+h(X(T))+\gamma(Y(0))\right] ,
\]
where $\varphi (\cdot )\in L_{\mathcal {F}}^2[0,T]$ and $\psi (\cdot
)\in L^2_{\mathcal{F}_{T}}[0,T].$ Basic assumptions imposed on $b,$
$\sigma,$ $g,$ $l,$ $h,$ $\gamma $ are stated as

(H3)
\begin{eqnarray*}
b(t,s,x,v) &:&\Delta\times R\times U\times \Omega\rightarrow R, \\
\sigma (t,s,x,v) &:&\Delta\times R\times U\times\Omega\rightarrow R, \\
g(t,s,x,y,z,v) &:&\Delta^c \times R\times R\times R\times U\times\Omega\rightarrow R, \\
l(s,x,y,v) &:&[0,T]\times R\times R\times U\times\Omega\rightarrow
R,
\end{eqnarray*}
$h(x):\Omega\times R\rightarrow R,$ $\gamma (x):\Omega\times
R\rightarrow R$. $b$, $\sigma,$ $g,$ $l,$ $h,$ and $\gamma $ are
continuously differentiable with respect to the variables. The
derivatives of $b,$ $\sigma,$ $g$ are bounded, the derivatives of
$l$ are bounded by $C(1+|x|+|y|+|v|)$ and the derivatives of $h$,
$\gamma$ with respect to $x$ are bounded by $C(1+|x|)$. Furthermore,
we assume that $g_{i}(t,s,x,y,z,v)$ is $\mathcal{B}(\Delta^c\times
R\times R\times R\times U)\otimes\mathcal{F}_{T}$-measurable such
that $t\mapsto g_{i}(t,s,x,y,z,v)$ is $\mathcal{F}$-progressively
measurable for all $(s,x,y,z,v)\in[0,T]\times R\times R\times
R\times U,$ $(i=y,z).$

 Given (H3) and $v\in\mathcal{U},$ we observe that
there exists a unique adapted M-solution $
(X(\cdot ),Y(\cdot ),Z(\cdot ,\cdot ))\in L_{\mathcal {F}}^2[0,T]\times L_{\mathcal {F}%
}^2[0,T]\times L^2(0,T;L_{\mathcal {F}}^2[0,T]) $ for above FBSVIE
(9) by what we mean that $X(\cdot )$ satisfies the forward equation
in (9) in the usual sense and $(Y(\cdot ),Z(\cdot ,\cdot ))$ is the
adapted M-solution of the backward form of (9). Both of the
equations in (9) are called the state equations.
The
optimal control problem is to minimize the cost function $J(v(\cdot
))$ over admissible controls. An admissible control $v(\cdot )$ is
called an optimal control if it attains the minimum.
\begin{remark}
A special case of the above optimal control problem was considered
in \cite{Y2} where
\[
J(u(\cdot ))=E\left[ \int_0^Tl(s,X(s),v(s))ds+h(X(T))\right],
\]
\end{remark}
and the coefficients are assumed to be independent of $\omega.$
\subsection{Variational equations and one convergence result}

Let $u(\cdot )$ be an optimal control and let $(X(\cdot ),Y(\cdot
),Z(\cdot ,\cdot ))$ be the corresponding M-solution of (9). Let
$v(\cdot )$ be such that $u(\cdot )+v(\cdot )\in \mathcal{U}$. Since
$\mathcal{U}$ is convex, then for any $0\leq \rho \leq 1,$ $u_\rho
=u(\cdot )+\rho v(\cdot )\in \mathcal{U}$. Let's us consider,
\begin{equation}
\left\{
\begin{array}{lc}
\xi (t)=\varphi_1 (t)+\displaystyle\int_0^tb_x(t,s,X^u(s),u(s))\xi
(s)ds+\displaystyle\int_0^t\sigma_{x}
(t,s,X^u(s),u(s))\xi(s)dW(s),   \\
\eta (t)=\psi _1(t)+\displaystyle\int_t^Tg_y(t,s,X^u(s),Y^u(s),Z^u(s,t),u(s))\eta (s)ds  \\
\quad \quad \quad
+\displaystyle\int_t^Tg_z(t,s,X^u(s),Y^u(s),Z^u(s,t),u(s))\zeta
(s,t)ds -\displaystyle\int_t^T\zeta(t,s)dW(s) ,
\end{array}
\right.
\end{equation}
where
\begin{equation}
\left\{
\begin{array}{lc}
 \varphi _1(t)
=\displaystyle\int_0^tb_v(t,s,X^u(s),u(s))v(s)ds+\displaystyle\int_0^t\sigma
_v(t,s,X^u(s),u(s))v(s)dW(s), \\
\psi _1(t) =\displaystyle\int_t^Tg_x(t,s,X^u(s),Y^u(s),Z^u(s,t),u(s))\xi (s)ds \\
\quad \quad \quad
\quad+\displaystyle\int_t^Tg_v(t,s,X^u(s),Y^u(s),Z^u(s,t),u(s))v(s)ds.
\end{array}
\right.
\end{equation}
The two equations in (10) are called variational equations.
Obviously under assumption (H3) we can find a unique M-solution $
(\xi (\cdot ),\eta (\cdot ),\zeta (\cdot ,\cdot ))\in L_{\mathcal {F}%
}^2[0,T]\times L_{\mathcal {F}}^2[0,T]\times L^2(0,T;L_{\mathcal
{F}}^2[0,T]), $ which is the unique adapted M-solution of FBSVIE
(10). We denote by $(X_\rho (\cdot ),Y_\rho (\cdot ),Z_\rho (\cdot
,\cdot ))$ the M-solutions of (9) corresponding to $u_\rho .$ We now
proceed to prove the relations
\begin{equation}
\left\{
\begin{array}{lc}
E \displaystyle\int^T_0|X_{\rho}(t)-X^u(t)|^2dt\rightarrow 0; \quad
\rho\rightarrow
0, \\
E\displaystyle\int^T_0|Y_{\rho}(t)-Y^u(t)|^2dt\rightarrow 0; \quad
\rho\rightarrow
0, \\
E\displaystyle\int^T_0\int^t_0|Z_{\rho}(t,s)-Z^u(t,s)|^2dsdt\rightarrow
0; \quad \rho\rightarrow 0.
\end{array}
\right.
\end{equation}
In fact, it follows from the denotation of $X_{\rho}$, together with
the forward equation in (9) that
\begin{eqnarray}
&& E\int^T_0e^{-\gamma t}|X_{\rho}(t)-X^u(t)|^2dt \nonumber \\
 &\leq& CE\int_0^Te^{-\gamma t} dt \int_0^t|X_{\rho}(s)-X^u(s)|^2ds
+CE\int_0^Te^{-\gamma t}dt \int_0^t|u_{\rho}(s)-u(s)|^2ds \nonumber \\
 &\leq &CE\int_0^T|X_{\rho}(s)-X^u(s)|^2ds\int_s^Te^{-\gamma t}dt+
 CE\int_0^T|u_{\rho}(s)-u(s)|^2ds\int_s^Te^{-\gamma t}dt
\nonumber \\
&\leq &\frac C\gamma E\int_0^Te^{-\gamma s}|X_{\rho}(s)-X^u(s)|^2ds+
\frac C\gamma E\int_0^Te^{-\gamma s}|u_{\rho}(s)-u(s)|^2ds,
\end{eqnarray}
where $\gamma$ is a positive constant depending on the upper bound
of all the derivatives. By choosing a $\gamma$ such that $\frac
C\gamma=\frac 1 2$, it leads to
\begin{eqnarray*}
E\int^T_0|X_{\rho}(t)-X^u(t)|^2dt\leq e^{\gamma
T}E\int^T_0e^{-\gamma t}|X_{\rho}(t)-X^u(t)|^2dt\rightarrow 0; \quad
\rho\rightarrow 0.
\end{eqnarray*}
Following the conclusion of Theorem 3.7 in \cite{Y2}, we observe
that
\begin{eqnarray}
&&E\int_0^T|Y_{\rho}(t)-Y^u(t)|^2dt+E\int_0^T\int_0^T|Z_{\rho}(t,s)-Z^u(t,s)|^2dsdt
\nonumber \\
&\leq& CE\int_0^T\left(\int_t^T|g'(t,s,Y_{\rho}(s),Z_{\rho}(s,t))
-g''(t,s,Y_{\rho}(s),Z_{\rho}(s,t))|ds\right)^2dt \nonumber
\\
&\leq& C
E\int_0^T|X_{\rho}(s)-X^u(s)|^2ds+CE\int_0^T|u_{\rho}(s)-u(s)|^2ds\rightarrow
0,\quad \rho\rightarrow 0,
\end{eqnarray}
where $g'(t,s,y,z)=g(t,s,X_{\rho}(s),y,z,u_{\rho}(s))$,
$g''(t,s,y,z)=g(t,s,X(s),y,z,u(s))$, $C$ is a constant depending on
the upper bound of all the derivatives. Thus we can get (12). For
$t,s\in [0,T],$ set
\begin{equation}
\left\{ \begin{array}{lc}
\widetilde{X}_\rho (t) =\rho ^{-1}(X_\rho (t)-X^u(t))-\xi (t), \\
\widetilde{Y}_\rho (t) =\rho ^{-1}(Y_\rho (t)-Y^u(t))-\eta (t), \\
\widetilde{Z}_\rho (s,t) =\rho ^{-1}(Z_\rho (s,t)-Z^u(s,t))-\zeta
(s,t).
\end{array}
\right.
\end{equation}
Using the similar method as (13), recalling the denotation of
$X_{\rho}$, we can deduce that
\[
E\int_0^Te^{-\alpha t}|\widetilde{X}_\rho (t)|^2dt\leq \frac C\alpha
E\int_0^Te^{-\alpha t}|\widetilde{X}_\rho (t)|^2dt+\varepsilon _\rho
,
\]
where $C$ is a constant depending on the upper bound of the derivatives, and $%
\varepsilon _\rho \rightarrow 0,$ $\rho \rightarrow 0$. Then we can
choose $\alpha $ such that $\frac C\alpha =\frac 1 2,$ and
\[
E\int_0^T|\widetilde{X}_\rho (t)|^2dt\leq e^{\alpha T}E\int_0^Te^{-\alpha t}|%
\widetilde{X}_\rho (t)|^2dt\leq 2e^{\alpha T} \varepsilon _\rho \rightarrow 0;%
\quad \rho \rightarrow 0.
\]
As to the term $\widetilde{Y}_\rho$, we arrive at
\begin{eqnarray*}
&&E\int_0^Te^{\beta t}|\widetilde{Y}_\rho (t)|^2dt+E\int_0^Te^{\beta
t}\int_t^T|\widetilde{Z}_\rho (t,s)|^2dsdt \\
&\leq &CE\int_0^Te^{\beta t}\int_t^T|\widetilde{X}_\rho
(s)|^2dsdt+CE\int_0^Te^{\beta t}\int_t^T|\widetilde{Y}_\rho (s)|^2dsdt \\
&&+\frac C\beta E\int_0^T\int_t^Te^{\beta s}|\widetilde{Z}_\rho
(s,t)|^2dsdt+Ce^{\beta T}\varepsilon _\rho ^{^{\prime }} \\
&\leq &\frac C\beta E\int_0^Te^{\beta s}|\widetilde{X}_\rho
(s)|^2ds+\frac C\beta E\int_0^Te^{\beta s}|\widetilde{Y}_\rho
(s)|^2ds+Ce^{\beta T}\varepsilon _\rho ^{^{\prime }},
\end{eqnarray*}
where $C$ is an constant depending on the upper bound of all the
derivative, and $\varepsilon _\rho ^{^{\prime }}\rightarrow 0,$
$\rho \rightarrow 0.$ Then we can choose a $\beta $ so that $\frac
C\beta <1,$ and
\begin{eqnarray*}
&&\ E\int_0^Te^{\beta t}|\widetilde{Y}_\rho
(t)|^2dt+E\int_0^Te^{\beta t}\int_t^T|\widetilde{Z}_\rho (t,s)|^2dsdt \\
\ &\leq &CE\int_0^Te^{\beta s}|\widetilde{X}_\rho (s)|^2ds+Ce^{\beta
T}\varepsilon _\rho ^{^{\prime }}.
\end{eqnarray*}
From above
\[
CE\int_0^Te^{\beta s}|\widetilde{X}_\rho (s)|^2ds\leq Ce^{\beta T}E\int_0^T|%
\widetilde{X}_\rho (t)|^2dt\rightarrow 0; \quad \rho \rightarrow 0,
\]
thus
\[
E\int_0^T|\widetilde{Y}_\rho (t)|^2dt\leq E\int_0^Te^{\beta t}|\widetilde{Y}%
_\rho (t)|^2dt\rightarrow 0;\quad  \rho \rightarrow 0.
\]
To sum up the argument above, we obtain:
\begin{lemma}
Let (H3) hold, then
\begin{equation}
\lim\limits_{\rho \rightarrow 0}E\int_0^T|\widetilde{X}_\rho
(s)|^2ds=0,\quad \lim\limits_{\rho \rightarrow
0}E\displaystyle\int_0^T|\widetilde{Y}_\rho (s)|^2ds=0.
\end{equation}
\end{lemma}
\subsection{A simple form of stochastic maximum principle}

In what follows, we make the following conventions with
$t,s\in[0,T],$ $v\in\mathcal{U},$
\begin{eqnarray*}
&& l_i^v(s)=l_i(s,X^v(s),Y^v(s),v(s)), i=x,y,v, \\
&&h_j^v(s,t)=h_j(s,t,X^v(t),v(t)),  j=x,v, h=b,\sigma,\\
&&g_k^v(s,t)=g_k(s,t,X^v(t),Y^v(t),Z^v(t,s),v(t)), k=x,y,z,v,
\end{eqnarray*}
where $(X^v,Y^v,Z^v)$ is the solution of (9) corespondent to $v$.
 In this subsection we assume that the cost function takes a simple form of $
J(u(\cdot ))=E\int_0^Tl(s,X(s),Y(s),u(s))ds.$ Since $u$ is an
optimal control, then $\rho ^{-1}[J(u+\rho v)-J(u)]\geq 0,$ and we
have the following variational inequality.
\begin{lemma}
Let (H3) hold, then
\begin{eqnarray}
E\int_0^Tl_x^u(s)\xi(s) ds+E\int_0^Tl_y^u(s)\eta(s)
ds+E\int_0^Tl_{v}^u(s)v(s)ds\geq 0,
\end{eqnarray}
 where $(X,Y,Z)$ is the unique
M-solution of FBSVIE (9) with $u$ being an optimal control.
\end{lemma}
\begin{proof}From the Lemma 4.1, we know
\begin{eqnarray*}
&&\ \rho ^{-1}E\int_0^T[l(s,X_\rho(s),Y_\rho(s) ,u_\rho(s) )-l(s,X^u(s),Y^u(s),u(s))]ds \\
&=&E\int_0^Tl_x(s,X^u(s)+\theta (X_\rho(s) -X^u(s)),Y_\rho(s) ,u_\rho(s) )\frac{X_\rho(s) -X^u(s)}%
{\rho} ds \\
&&+E\int_0^Tl_y(s,X^u(s),Y^u(s)+\theta (Y_\rho(s) -Y^u(s)),u_\rho(s) )\frac{Y_\rho(s) -Y^u(s)}{\rho} ds \\
&&+ E\int_0^Tl_u(s,X^u(s),Y^u(s),u(s)+\theta (u_\rho(s) -u(s)))v(s)ds \\
&\rightarrow& E\int_0^Tl_x(s,X^u(s),Y^u(s),u(s))\xi(s) ds+E\int_0^Tl_y(s,X^u(s),Y^u(s),u(s))\eta(s) ds \\
&&+E\int_0^Tl_v(s,X^u(s),Y^u(s),u(s))v(s)ds.
\end{eqnarray*}
Thus the conclusion follows.
\end{proof}

For deriving the maximum principle, we introduce the following two
adjoint equations:
\begin{equation}
\left\{ \begin{array}{lc} P(t)
=l_y^u(t)+\displaystyle\int_0^tg_y^u(s,t)P(s)ds+\displaystyle\int_0^tg_z^u(s,t)P(s)dW(s),\\
Q(t)
=l_x^u(t)+\displaystyle\int_0^tg_x^u(s,t)P(s)ds+\displaystyle\int_t^Tb_x^u(s,t)Q(s)ds\\
\quad \quad \quad \quad +\displaystyle\int_t^T\sigma
_x^u(s,t)R(s,t)ds-\displaystyle\int_t^TR(t,s)dW(s).
\end{array}
\right.
\end{equation}
Obviously the above FBSVIE admits a unique M-solution $(P(\cdot
),Q(\cdot ),R(\cdot ,\cdot ))$ under assumption (H3). Note that
$g_{y}$ and $g_{z}$ are non-anticipated processes under (H3). The
later proposition is the so-called dual principle for linear
stochastic Volterra integral equation, the proof of which can be
found in \cite{Y2}.
\begin{proposition}
Let $A_i:\triangle\times \Omega\rightarrow R$ $%
(i=1,2)$ be $\mathcal {B}(\triangle)\otimes \mathcal
{F}_{T}$-measurable such that $s\rightarrow A(t,s)$ is $\mathcal
{F}$-progressively measurable for all $t\in[0,T],$ furthermore, we
assume that they are two bounded processes, $\varphi (\cdot )\in
L_{\mathcal {F}}^2[0,T]$ and $\psi (\cdot )\in L^2_{\mathcal
{F}_{T}}[0,T].$ Let $\xi (\cdot )\in L_{\mathcal {F}}^2[0,T]$ be the
solution of FSVIE:
\[
\xi (t)=\varphi (t)+\int_0^tA_1(t,s)\xi (s)ds+\int_0^tA_2(t,s)\xi
(s)dW(s),\quad t\in [0,T],
\]
and $(Y(\cdot ),Z(\cdot ,\cdot ))$ be the adapted M-solution to the
following BSVIE, $\forall t\in [0,T],$%
\[
Y(t)=\psi
(t)+\int_t^T\{A_1(s,t)Y(s)+A_2(s,t)Z(s,t)\}ds-\int_t^TZ(t,s)dW(s).
\]
Then the following relation holds:
\[
E\int_0^T\xi (t)\psi (t)dt=E\int_0^T\varphi (t)Y(t)dt.
\]
\end{proposition}
We now assert:
\begin{theorem}
Let $u(\cdot )$ be an optimal control and $(X^u(\cdot ),Y^u(\cdot
),Z^u(\cdot ,\cdot ))$ be the corresponding M-solution of FBSVIE
(9). Then we have, $\forall
v\in U,$%
\[
H(X^u(t),Y^u(t),Z^u(t,\cdot),u(t),P(t),Q(t),R(\cdot,t))\cdot
(v-u(t))\geq 0,\quad a.e.,  a.s.
\]
where
\begin{eqnarray*}
&& H(X^u(t),Y^u(t),Z^u(t,\cdot),u(t),P(t),Q(t),R(\cdot,t)) \\
&=&l_v^u(t)+E^{\mathcal{F}_{t}}\int_t^TQ(s)b_v^u(s,t)ds
+E^{\mathcal{F}_{t}}\int_t^TR(s,t)\sigma _v^u(s,t)ds
+\int_0^tg_v^u(s,t)P(s)ds
\end{eqnarray*}
Here $(P,Q,R)$ is the unique M-solution of FBSVIE (18).
\end{theorem}
\begin{proof}From the forward form in (10), the backward form in (18) and Proposition 4.1 above, we know
that
\begin{eqnarray}
&&E\int_0^T\xi(t) l_x^u(t)dt+E\int_0^T\xi(t) \int_0^tg_x^u(s,t)P(s)dsdt   \nonumber \\
&=&E\int_0^TQ(t)dt\int_0^tb_v^u(t,s)v(s)ds+E\int_0^TQ(t)dt\int_0^t\sigma
_v^u(t,s)v(s)dW(s)
\nonumber \\
&=&E\int_0^TQ(t)dt\int_0^tb_v^u(t,s)v(s)ds+E\int_0^T\int_0^tR(t,s)\sigma
_v^u(t,s)v(s)dsdt
\nonumber \\
&=&E\int_0^Tv(s)K(s)ds ,
\end{eqnarray}
where
\begin{eqnarray*}
K(s) &=&\int_s^T\{Q(t)b_v^u(t,s)+R(t,s)\sigma _v^u(t,s)\}dt.
\end{eqnarray*}
Similarly from the backward form in (10), the forward form in (18)
and Proposition 4.1, we know
\begin{eqnarray}
E\int_0^T\eta(t) l_y^u(t)dt &=&E\int_0^TP(t)dt\int_t^Tg_x^u(t,s)\xi
(s)ds
+E\int_0^TP(t)dt\int_t^Tg_v^u(t,s)v(s)ds  \nonumber \\
&=&E\int_0^T\xi (t)dt\int_0^tg_x^u(s,t)P(s)ds
+E\int_0^Tv(s)ds\int_0^sg_v^u(t,s)P(t)dt.\nonumber \\
\end{eqnarray}
It follows form (19) and (20),
\begin{eqnarray*}
&&E\int_0^T\xi(t) l_x^u(t)dt+E\int_0^T\eta(t) l_y^u(t)dt \\
&=&E\int_0^Tv(s)ds\left[
\int_s^TQ(t)b_v^u(t,s)dt+\int_s^TR(t,s)\sigma
_v^u(t,s)dt+\int_0^sg_v^u(t,s)P(t)dt\right].
\end{eqnarray*}
From the variational inequality (17) we have
\begin{eqnarray*}
0 &\leq &E\int_0^Tl_x^u(t)\xi(t)dt+E\int_0^Tl_y^u(t)\eta (t)dt+E\int_0^Tl_v^u(t)v(t)dt \\
&=&E\int_0^Tv(t)L(t)dt,
\end{eqnarray*}
where
\begin{eqnarray*}
L(t) &=&l_v^u(t)+\int_t^TQ(s)b_v^u(s,t)ds +\int_t^TR(s,t)\sigma
_v^u(s,t)ds+\int_0^tg_v^u(s,t)P(s)ds.
\end{eqnarray*}
The proof is complete.
\end{proof}
\subsection{A general stochastic maximum principle}
In this subsection we will give a new maximum principle, while the
cost function is a more general form
\[
J(v(\cdot ))=E\left[
\int_0^Tl(s,X(s),Y(s),v(s))ds+h(X(T))+\gamma(Y(0))\right] .
\]
It can be easily checked that $
E\int_0^Tl_x^u(s)\xi(s)ds+E\int_0^Tl_y^u(s)\eta(s)
ds+E\int_0^Tl_{v}^u(s)v(s)ds +Eh_x(X^u(T))\xi
(T)+E\gamma_{y}(Y^u(0))\eta(0)\geq 0.$ In fact, the definition of
$\widetilde{X}_{\rho}$ implies $E|\widetilde{X}_{\rho}(T)|^2\leq
\delta_1(\rho)+CE\int_0^T|\widetilde{X}_{\rho}(s)|^2ds,$ with
$\delta_1(\rho)\rightarrow0, \rho\rightarrow0,$ $C$ is a constant
depending on the upper bound of all the derivatives. Recalling the
result in Lemma 4.1 we obtain
$E|\widetilde{X}_{\rho}(T)|^2\rightarrow0, \rho\rightarrow0.$
Similarly by the form of $\widetilde{Y}_{\rho}$, it follows that
$$E|\widetilde{Y}_{\rho}(0)|^2\leq
CE\int_0^T|\widetilde{X}_{\rho}(s)|^2ds+CE\int_0^T|\widetilde{Y}_
{\rho}(s)|^2ds+CE\int_0^T|\widetilde{Z}_{\rho}(s,0)|^2ds+\delta_2(\rho),$$
with $\delta_2(\rho)\rightarrow0, \rho\rightarrow0.$ By the
definition of M-solution in the previous, it does no matter what
value of $\zeta(s,0)$ is as long as it is a deterministic function
on $s\in[0,T]$. In particular, if $E\int_0^t|D_s\eta(t)|^2ds<\infty$
(here $D$ is a malliavin operator, see \cite{D} for more detailed
accounts for malliavin calculus), then by Ocone-Clark formula (see
\cite{D}) and the definition of M-solution, we obtain
$$\eta(t)=E\eta(t)+\int_0^t\zeta(t,s)dW(s)=E\eta(t)+\int_0^tE^{\mathcal{F}_{s}}D_s\eta(t)dW(s),$$ thus we have
$\zeta(t,s)=E^{\mathcal{F}_{s}}D_s\eta(t),$ then without loss of
generality we can determine $\zeta(s,0)$ by $E\eta(s).$ Similarly
$\widetilde{Z}_{\rho}(s,0)=E\widetilde{Y}_{\rho}(s),$ and this leads
to $E|\widetilde{Y}_{\rho}(0)|^2\rightarrow 0,$ with
$\rho\rightarrow 0.$

 Summing up, we
finally obtain $ \rho^{-1}E(h(X_{\rho}(T))-h(X^u(T)))\rightarrow
Eh_{x}(X^u(T))\xi(T),$ and
$\rho^{-1}E(\gamma(Y_{\rho}(0))-\gamma(Y^u(0)))\rightarrow
E\gamma_{y}(Y^u(0))\eta(0),$ with $E|\xi(T)|^2<\infty,\quad
E|\eta(0)|^2<\infty,$ which is easy to validate.

It follows from the martingale representation theorem that there
exists a unique process $\pi (s)\in L_{\mathcal{F}}^2[0,T]$ so that
$ h_x(X^u(T))=Eh_x(X^u(T))+\int_0^T\pi (s)dW(s); $ then
\begin{eqnarray*}
&&\ Eh_x(X^u(T))\xi (T) \\
&=&Eh_x(X^u(T))\left[ \int_0^Tb_v^u(T,s)v(s)ds+\int_0^T\sigma
_v^u(T,s)v(s)dW(s)\right]\\
&&+Eh_x(X^u(T))\left[\int_0^Tb_x^u(T,s)\xi(s)
ds+\int_0^T\sigma _x^u(T,s)\xi(s)dW(s)\right]  \\
&=&E\int_0^Tb_v^u(T,s)h_x(X^u(T))v(s)ds+E\int_0^T\pi(s) \sigma _v^u(T,s)v(s)ds \\
&&+E\int_0^Tb_x^u(T,s)\xi(s)h_x(X^u(T))ds+E\int_0^T\sigma
_x^u(T,s)\xi(s)\pi(s)ds.
\end{eqnarray*}
On the other hand, using the fact that $
E\gamma_y(Y^u(0))\int_0^T\zeta(0,s)dW(s)=0$, one gets
\begin{eqnarray}
E\gamma_y(Y^u(0))\eta(0)&=&E\int_0^Tg_x^u(0,s)\gamma_y(Y^u(0))\xi(s)ds
+E\int_0^Tg_v^u(0,s)\gamma_y(Y^u(0))v(s)ds \nonumber \\
&+& E\int_0^Tg_y^u(0,s)\gamma_y(Y^u(0))\eta (s)ds+
E\int_0^Tg_z^u(0,s)\gamma_y(Y^u(0))E\eta(s)ds.\nonumber \\
\end{eqnarray}
In this case, FBSVIE (18) is replaced by
\begin{equation}
\left\{
\begin{array}{lc}
P(t)=l_y^u(t)+g_y^u(0,t)\gamma_y(Y^u(0))+\gamma_y(Y^u(0))Eg_z^u(0,t) \\
\quad \quad \quad \quad + \displaystyle\int_0^tg_y^u(s,t)P(s)ds
+\displaystyle\int_0^tg_z^u(s,t)P(s)dW(s),\\
 Q(t)
=l_x^u(t)+b_x^u(T,t)h_x(X^u(T))+\sigma _x^u(T,t)\pi(t)
 +g_x^u(0,t)\gamma_y(Y^u(0))+\displaystyle\int_0^tg_x^u(s,t)P(s)ds\\
\quad \quad \quad \quad+\displaystyle\int_t^Tb_x^u(s,t)Q(s)ds
+\displaystyle\int_t^T\sigma
_x^u(s,t)R(s,t)ds-\displaystyle\int_t^TR(t,s)dW(s).
\end{array} \right.
\end{equation}
So by a similar proof as Theorem 4.1 we get a more general
stochastic maximum principle.
\begin{theorem}
Let $u(\cdot )$ be an optimal control and $(X(\cdot ),Y(\cdot
),Z(\cdot ,\cdot ))$ be the corresponding M-solution of FBSVIE (9).
Then we have, $\forall
v\in U,$%
\begin{eqnarray}
H(t,X^u(t),Y^u(t),Z^u(t,\cdot),u(t),P(t),Q(t),R(\cdot,t))\cdot
(v-u(t))\geq 0, \quad a.e., a.s.
\end{eqnarray}
where
\begin{eqnarray*}
&&H(t,X(t),Y(t),Z(t,\cdot),u(t),P(t),Q(t),R(\cdot,t)) \\
&=&l_v^u(t)+b_v^u(T,t)E^{\mathcal{F}_{t}}h_x(X^u(T))+\sigma _v^u(T,t)\pi (t)+g_v^u(0,t) \\
&&+\int_0^tg_v^u(s,t)P(s)ds +E^{\mathcal{F}_{t}}\int_t^TR(s,t)\sigma
_v^u(s,t)ds+E^{\mathcal{F}_{t}}\int_t^TQ(s)b_v^u(s,t)ds
\end{eqnarray*}
where $(P,Q,R)$ is the unique M-solution of FBSVIE (22).
\end{theorem}

If we define
\begin{eqnarray*}
\mathcal{H}(t,X^u(t),Y^u(t),Z^u(t,\cdot),u(t),P(t),Q(t),R(\cdot,t),v)=-H\cdot
v,
\end{eqnarray*}
then (23) can be rewritten as
\begin{eqnarray}
&&\mathcal{H}(t,X^u(t),Y^u(t),Z^u(t,\cdot),u(t),P(t),Q(t),R(\cdot,t),u(t)) \nonumber \\
&=&\max_{v\in
U}\mathcal{H}(t,X^u(t),Y^u(t),Z^u(t,\cdot),u(t),P(t),Q(t),R(\cdot,t),v)
\end{eqnarray}
We call $\mathcal{H}$ the Hamiltonian of the optimal control problem
of FBSVIEs, and (24) the maximum principle condition.

We would like to conclude this section by giving a application, that
is, a linear-quadratic (LQ for short) problem of BSVIEs. The linear
BSVIE is of the form
\begin{equation}
Y(t)=\psi
(t)+\int_t^T[l_1(t,s)Y(s)+l_2(t,s)v(s)+l_3(t,s)Z(s,t)]ds-\int_t^TZ(t,s)dW(s),
\end{equation}
while the cost functional $J(\psi (\cdot ),v(\cdot ))$ associated
with the terminal condition $\psi (\cdot )$ and control $v(\cdot)$
is given by
\[
J(\psi (\cdot ),v(\cdot ))=\frac
12E\int_0^T[Q'(t)Y^2(t)+R'(t)v^2(t)]dt+\frac 12 EG'Y^2(0).
\]
The linear-quadratic control problem is to minimize the cost
function over admissible controls. Necessary assumptions will be in
force in the following.

(H4) Let $l_i:\Delta^c\rightarrow R,$ ($i=1,2,3)$ be three
continuous bounded processes such that $s\rightarrow l_i(t,s)$ is
$\Bbb{F}$ adapted for $t\in[0,T]$. $Q'$ and $R$ are bounded and
non-negative adapted processes, moreover, $R'(t)>\delta,$ where
$\delta$ is a positive constant, $G'$ is a non-negative bounded
random variable,
 $\psi (\cdot )\in L_{
\mathcal{F}_T}^2[0,T].$ In addition, assume $U$ is also closed.

Obviously (H4) is sufficient for the finiteness of the above
linear-quadratic problem. Following the idea of Theorem 5.2 in
Chapter 2 of \cite{YZ}, we are ready to present a existence theorem
of the optimal control.

\begin{lemma}
Let (H4) hold, then there exists a $u(\cdot)\in\mathcal{U}$ such
that $J(\psi,u)=\inf\limits_{v\in \mathcal{U}}J(\psi,v).$
\end{lemma}
\begin{proof}
Let $\psi (\cdot )$ be fixed, and $u_j(\cdot )\in
L_{\Bbb{F}}^2[0,T]$ be a minimizing sequence of LQ problem, that is
\begin{eqnarray}
\lim\limits_{j\rightarrow \infty }J(\psi (\cdot ),u_j(\cdot
))=\inf_{v(\cdot )\in \mathcal{U}}J(\psi (\cdot ),v(\cdot )).
\end{eqnarray}
Let $(Y^j,Z^j)$ be the state processes corresponding to $u_j(\cdot
)$. It follows from (26) that there exists a constant $M$ such that
$J(\psi (\cdot ),u_j(\cdot ))\leq M$ for any $j\geq 1.$
Additionally, $J(\psi (\cdot
),u_j(\cdot ))\geq \delta E\int_0^T|u_j(t)|^2dt,$ so we have $%
E\int_0^T|u_j(t)|^2dt\leq \frac M\delta .$ Consequently, there is a
subsequence, which is still labeled by $u_j(\cdot )$, such that,
\[
u_j(\cdot )\rightarrow u^{\prime }(\cdot )\text{, weakly in }L_{\Bbb{F}%
}^2[0,T],
\]
By Mazur's theorem, we have a sequence of convex combinations
\[
\widehat{u}_j(\cdot )=\sum_{i\geq 1}\alpha _{ij}u_{i+j}(\cdot ),\text{ }%
\alpha _{i,j}\geq 0,\sum_{i\geq 1}\alpha _{ij}=1,
\]
such that
\[
\widehat{u}_j(\cdot )\rightarrow u^{\prime }(\cdot ),\text{ strongly in }L_{%
\Bbb{F}}^2[0,T].
\]
Since the set $U$ is convex and closed, it follows that $u^{\prime
}(\cdot )\in \mathcal{U}.$ On the other hand, the Theorem 3.7 in
\cite{Y2} leads to
\[
E\int_0^T|\widehat{Y}_j(t)-Y^{\prime }(t)|^2dt\leq CE\int_0^T\left[
\int_t^T(\widehat{u}_j(s)-u^{\prime }(s))ds\right] ^2dt\leq CE\int_0^T|%
\widehat{u}_j(s)-u^{\prime }(s)|^2ds
\]
i.e., $\widehat{Y}_j(\cdot )\rightarrow Y^{\prime }(\cdot ),$
strongly in $L_{\Bbb{F}}^2[0,T].$ By the convexity of the generator
for (25),
\begin{eqnarray*}
J(\psi (\cdot ),u^{\prime }(\cdot )) &=&\lim\limits_{j\rightarrow
\infty
}J(\psi (\cdot ),\widehat{u}_j(\cdot )) \\
&\leq &\lim\limits_{j\rightarrow \infty }\sum_{i\geq 1}\alpha
_{ij}J(\psi (\cdot ),u_{i+j}(\cdot ))=\inf_{u(\cdot )\in
\mathcal{U}}J(\psi (\cdot ),u(\cdot )),
\end{eqnarray*}
which means that $u^{\prime }(\cdot )$ is an optimal control.
\end{proof}

In this setting, the maximum principle condition can be written as
\begin{eqnarray}
&&-R'(t)u^2(t)-l_2(0,t)u(t)-u(t)\int_0^tl_2(s,t)P(s)ds \nonumber \\
 &&\geq
-R'(t)u(t)v-l_2(0,t)v-v\int_0^tl_2(s,t)P(s)ds,
\end{eqnarray}
with $v\in U$, and this leads to
$R'(t)u(t)+l_2(0,t)+\int_0^tl_2(s,t)P(s)ds=0,$ i.e.,
$u(t)=R'^{-1}(t)[l_2(0,t)+\int_0^tl_2(s,t)P(s)ds],$ where
\begin{eqnarray}
P(t)=Q'(t)+[l_1(0,t)+l_3(0,t)]G'Y(0)+\int_0^tl_1(s,t)P(s)ds+\int_0^tl_3(s,t)P(s)dW(s).
\end{eqnarray}
Hence $u(t)$ is the only control which satisfies the necessary
conditions of optimality. By Lemma 4.3, it must be the unique
optimal control. So we have
\begin{theorem}
Let (H4) hold, there is a unique optimal control $u(\cdot)$ for the
linear-quadratic control problem. Moreover, $u$ has the
representation: $u(t)=R'^{-1}(t)[l_2(0,t)+\int_0^tl_2(s,t)P(s)ds],$
where $P(s)$ satisfies (28).
\end{theorem}

By the form of the optimal control, we deduce that the optimal
control indeed a linear state feedback of the entire past history of
the state process $P(\cdot)$ instead of being a feedback of the
current state, which is similar to the result for linear-quadratic
control of BSDEs, see p.6-p.7 in \cite{LZ}. Substituting the
representation of $u(\cdot)$ into (25), together with equation (28),
we get the following coupled FBSVIE
\begin{equation}
\left\{
\begin{array}{lc}
P(t)=Q'(t)+[l_1(0,t)+l_3(0,t)]G'Y(0)+\displaystyle\int_0^tl_1(s,t)P(s)ds+\displaystyle\int_0^tl_3(s,t)P(s)dW(s),\\
Y(t)=\psi(t)+\displaystyle\int_t^Tl_2(s,t)R'^{-1}(s)l_2(0,s)ds+\displaystyle\int_t^Tl_1(s,t)Y(s)ds\\
\quad \quad \quad +
\displaystyle\int_t^Tl_2(s,t)\int_0^sl_2(u,s)P(u)duds+\displaystyle\int_t^Tl_3(s,t)Z(s,t)ds-\displaystyle\int_t^TZ(t,s)dW(s),
\end{array} \right.
\end{equation}
Given (H4), by the unique existence of optimal control $u(\cdot)$,
we observe that FBSVIE (29) admits a unique M-solution
$(Y(\cdot),Z(\cdot,\cdot),P(\cdot))$ by which means that $P(\cdot)$
solves the forward equation of (29) in the It\^{o} sense and
$(Y(\cdot),Z(\cdot,\cdot))$ is the unique M-solution of the backward
equation in (29). Consequently,
\begin{theorem}
Let (H4) hold, then FBSVIE (29) admits a unique M-solution.
\end{theorem}

\section{Application in fincance}
In this section, we illustrate the maximum principle above by
studying the risk minimization problem in finance. Such kind of
problem was studied by Mataramvura and {\O}ksendal \cite{MO} by
formulating it as a zero-sum stochastic differential game. Recently
{\O}ksendal and Sulem \cite{OS} also investigated this risk
minimization problem via g-expectations. In this paper, we will
consider the problem by means of the maximum principle in the
previous. Some closed forms of the optimal solution are derived,
which are consistent with the results in \cite{OS} or \cite{WXZ}.

We consider a market with two investment possibilities, which are
traded continuously until the fixed finite horizon $T,$ is reached.
One investment is described by
\[
dS_0(t)=S_0(t)\rho (t)dt,\text{ }S_0(0)=s_0.
\]
The other financial instrument is described by
\[
dS_1(t)=S_1(t)\alpha (t)+S_1(t)\beta (t)dW(t),\text{ }S_1(0)=s_1.
\]
Suppose the interest rate $\rho (\cdot )$ is nonnegative and bounded
deterministic function, the stock-appreciation rate $\alpha (\cdot
)$ and the stock-volatility $\beta (\cdot )$ are nonnegative and
bounded adapted processes. Moreover, $\beta ^{-1}(\cdot )$ and
$(\alpha (\cdot )-\rho (\cdot ))^{-1}$ exist and bounded. The wealth
process $X(\cdot )$ satisfies
\begin{eqnarray}
dX(t) =[\rho (t)X(t)+v(t)(\alpha (t)-\rho (t))]dt+v(t)\beta
(t)dW(t),
\end{eqnarray}
with $ X(0)=x>0,$ thereby the solution of the wealth equation can be
given by
\[
X(t)=e^{\int_0^t\rho (s)ds}x+\int_0^te^{\int_s^t\rho
(u)du}[v(s)(\alpha (s)-\rho (s))ds+v(s)\beta (s)dW(s)].
\]
A portfolio $v(\cdot )$, representing the amount invested in the
risk asset, is said to be admissible if $v(\cdot )\in \mathcal{U}.$
In the following the BSVIE that we are going to investigate is
\begin{eqnarray}
Y(t)=-\psi
(t)+\int_t^T[r(s)Y(s)+k_1(t,s)Z(s,t)+k_2(t,s)]ds-\int_t^TZ(t,s)dW(s),
\end{eqnarray}
where
\[
\psi (t)=h(X(T))+\int_t^T[l_1(t,s)X(s)+l_2(t,s)v(s)]ds.
\]
Here $r,$ $l_i$ are bounded deterministic functions, $k_i$ is a
process such that $s\rightarrow k_i(t,s)$ is $\Bbb{F}$-progressively
measurable for almost $t\in [0,T],$ $k_1$ is bounded, $k_2$ has the
same assumption with $g_0$ in (H1). $h$ is utility function of the
terminal wealth that satisfies (H3) and $Eh^2(X(T))<\infty $ (see
the special cases below). If we define $\varrho(t;\psi (\cdot
))=Y(t),$ then by Theorem 3.6 in Yong \cite{Y3}, $\varrho$ is a
dynamic coherent risk measure. In order to show the result in a more
explicit way, we would like to consider the special static case of
$\varrho(0,\psi(\cdot ))=Y(0)$, and denote the cost functional by
\begin{eqnarray*}
J(v(\cdot )) &=&Y(0)=-E\psi
(0)+E\int_0^T[r(s)Y(s)+k_1(0,s)Z(s,0)+k_2(0,s)]ds
\\
&=&-Eh(X(T))-E\int_0^T[l_1(0,s)X(s)+l_2(0,s)v(s)]ds \\
&&+E\int_0^T[r(s)Y(s)+k_1(0,s)EY(s)+k_2(0,s)]ds.
\end{eqnarray*}
We want to find $u\in\mathcal{U}$ such that $J(u(\cdot
))=\inf\limits_{v\in \mathcal{U}}J(v(\cdot )).$
With the notation in the previous, we obtain $\gamma (y)=0$, and $
l(s,x,y,v)=[r(s)+Ek_1(0,s)]y-l_1(0,s)x-l_2(0,s)v+k_2(0,s). $ As to
the coefficients in both (30) and (31),
\begin{eqnarray*}
b_x(t,s,x,v) &=&\rho (s),\sigma _x(t,s,x,v)=0,b_v(t,s,x,v)=\alpha
(s)-\rho
(s), \\
\sigma _v(t,s,x,v) &=&\beta (s),\text{ }g_x^{\prime
}(t,s)=l_1(t,s),\text{ }
\\
g_v^{\prime }(t,s) &=&l_2(t,s),g_y^{\prime }(t,s)=r(s),g_z^{\prime
}(t,s)=k_1(t,s),
\end{eqnarray*}
where
\[
g_i^{\prime }(t,s)=g_i(t,s,x,y,z,v),\text{ }i=x,v,y,z.
\]
Then the Hamilton function is the form of
\begin{eqnarray*}
&&\mathcal{H}(t,X^u(t),Y^u(t),Z^u(t,\cdot ),u(t),P(t),Q(t),R(\cdot ,t),v) \\
&=&-v\left[ (\alpha (t)-\rho(t))E^{\mathcal{F}_t}h_x(X^u(T))+\beta
(t)\pi
(t)\right] \\
&&-v\left[ \int_0^tl_2(s,t)P(s)ds+\beta (t)E^{\mathcal{F}_t}%
\int_t^TR(s,t)ds+(\alpha (t)-\rho
(t))E^{\mathcal{F}_t}\int_t^TQ(s)ds\right],
\end{eqnarray*}
and the adjoint equation
\begin{equation}
\left\{
\begin{array}{c}
P(t)=[r(t)+Ek_1(0,t)]+r(t)\displaystyle\int_0^tP(s)ds+\displaystyle\int_0^tk_1(s,t)P(s)dW(s),
\\
Q(t)=-l_1(0,t)-\rho (t)h_x(X^u(T))+\displaystyle\int_0^tl_1(s,t)P(s)ds \\
+\displaystyle\int_t^T\rho (t)Q(s)ds-\displaystyle\int_t^TR(t,s)dW(s), \\
h_x(X^u(T))=Eh_x(X^u(T))+\displaystyle\int_0^T\pi (s)dW(s).
\end{array}
\right.
\end{equation}
Since $\mathcal{H}$ is a linear function in $v,$ then the
coefficient of $v$ vanishes, i.e.,
\begin{eqnarray}
0=(\alpha (t)-\rho (t))E^{\mathcal{F}_t}h_x(X^u(T))+\beta (t)\pi
(t)+M(t),
\end{eqnarray}
where
\begin{eqnarray}
M(t) &=&\int_0^tl_2(s,t)P(s)ds+\beta (t)E^{\mathcal{F}_t}\int_t^TR(s,t)ds+(%
\alpha (t)-\rho (t))E^{\mathcal{F}_t}\int_t^TQ(s)ds.\nonumber \\
\end{eqnarray}
The lemma below is essentially similar to Theorem A.2 in the
appendix of \cite{OS}. For readers' convenience, we present a proof
here.
\begin{lemma}
Consider the equation, $\alpha _1(t)E^{\mathcal{F}_t}\xi +\beta
_1(t)\theta (t)=\zeta (t),$ where $\xi $ is a
$\mathcal{F}_T$-measurable random variable satisfying $\xi =E\xi
+\int_0^T\theta (s)dW(s).$ Assume $\beta _1^{-1}(t)$ exists, $\zeta
(t)$ is a adapted process, then $\xi$ must be the form of
\[
\xi =e^{-A_1(T)}E\xi +e^{-A_1(T)}\int_0^Te^{A_1(s)}\beta
_1^{-1}(s)\zeta (s)dW(s),
\]
where $A_1$ is given by
\begin{eqnarray}
A_1(t)=\int_0^t\beta _1^{-1}(s)\alpha _1(s)dW(s)+\frac
12\int_0^t\beta _1^{-2}(s)\alpha _1^2(s)ds.
\end{eqnarray}
\end{lemma}
\begin{proof}We denote by $P(t)=E^{\mathcal{F}%
_t}\xi$, therefore $P(t)=P(0)+\int_0^t\theta (s)dW(s)$. On the other
hand,
\[
\theta (t)=\beta _1^{-1}(t)[\zeta (t)-\alpha_1(t)P(t)],
\]
so
\[
P(t)=P(0)+\int_0^t\beta _1^{-1}(s)[\zeta (s)-\alpha _1(s)P(s)]dW(s).
\]
Since we can rewrite $P$ by
\[
P(t)=e^{-A_1(t)}P(0)+e^{-A_1(t)}\int_0^te^{A_1(s)}\beta
_1^{-1}(s)\zeta (s)dW(s),
\]
with $A_1$ given by (35), thereby
\[
\xi =e^{-A_1(T)}E\xi +e^{-A_1(T)}\int_0^Te^{A_1(s)}\beta
_1^{-1}(s)\zeta (s)dW(s).
\]
The conclusion follows clearly.
\end{proof}
\begin{remark}
There are two things worthy to point out. On the one hand, if we
assume that
$h_x(X^u(T))\in \Bbb{D}_{1,2},$ see \cite{D}, then by Ocone-Clark formula, $\pi(t)=E^{\mathcal{F}%
_t}D_th_x(X^u(T)),$ then (33) can be rewritten as
\[
(\alpha (t)-\rho (t))E^{\mathcal{F}_t}h_x(X^u(T))+\beta (t)E^{\mathcal{F}%
_t}D_th_x(X^u(T))+M(t)=0.
\]
It is a linear inhomogeneous Malliavin-differential type equation in
the unknown random variable $h_x(X^u(T)),$ which can also seen in
\cite{MOZ}and \cite{OS}. On the other hand, if $\zeta =0,$ there are
infinite random variables satisfying the equation in Lemma 5.1. For
example, if $\xi_1$ satisfies it, so does $c\xi_1$, with $c$ being a
constant.
\end{remark}
To sum up, we give
\begin{theorem}
Suppose $u$ is an optimal portfolio of the above risk minimizing
problem, then $u(\cdot )$ must satisfies
$$h_x(X^u(T))=e^{-A(T)}Eh_x(X^u(T))-e^{-A(T)}\int_0^Te^{A(s)}\beta^{-1}
(s)M(s)dW(s),$$ with $X^u(T)$ be the terminal wealth corresponding
to $u,$ $M(t)$ given by (34), and
$$A(t)=\int_0^t\beta ^{-1}(s)(\alpha (s)-\rho (s))dW(s)+\frac
12\int_0^t\beta ^{-2}(s)(\alpha (s)-\rho (s))^2ds.$$
\end{theorem}

In order to express the explicit form of $u$, some more assumptions
are required. Let $l_i=0,$ $(i=1,2),$ then the Hamilton function
becomes
\begin{eqnarray*}
\mathcal{H} &=&-v[(\alpha (t)-\rho
(t))E^{\mathcal{F}_t}h_x(X^u(T))+\beta
(t)\pi (t)] \\
&&-v\left( \beta (t)E^{\mathcal{F}_t}\int_t^TR(s,t)ds+(\alpha
(t)-\rho (t))E^{\mathcal{F}_t}\int_t^TQ(s)ds\right) ,
\end{eqnarray*}
where
\[
\left\{
\begin{array}{c}
Q(t)=-\rho (t)h_x(X^u(T))+\displaystyle\int_t^T\rho (t)Q(s)ds-\displaystyle\int_t^TR(t,s)dW(s), \\
h_x(X^u(T))=Eh_x(X^u(T))+\displaystyle\int_0^T\pi (s)dW(s),
\end{array}
\right.
\]
therefore, the optimal portfolio $u$ satisfies
\begin{eqnarray}
0 &=&(\alpha (t)-\rho (t))E^{\mathcal{F}_t}h_x(X^u(T))+\beta (t)\pi (t)\nonumber \\
&&+\beta (t)E^{\mathcal{F}_t}\int_t^TR(s,t)ds+(\alpha (t)-\rho (t))E^{%
\mathcal{F}_t}\int_t^TQ(s)ds.
\end{eqnarray}
As a consequence, by solving the above simple BSVIE, we deduce that
$\forall (t,s)\in
\Delta ,$ i.e., $0\leq s<t\leq T,$%
\begin{eqnarray}
Q(t)=-\rho (t)e^{\int_t^T\rho (u)du}\cdot E^{\mathcal{F}%
_t}h_x(X^u(T)),R(t,s)=-\rho (t)e^{\int_t^T\rho (u)du}\pi (s).
\end{eqnarray}
Substituting (37) into (36), we arrive at
\begin{eqnarray}
(\alpha (t)-\rho (t))E^{\mathcal{F}_t}h_x(X(T))+\beta (t)\pi (t)=0.
\end{eqnarray}
Recalling Lemma 5.1, we can express $h_x(X^u(T))$ as
$h_x(X^u(T))=e^{-A(T)}Eh_x(X^u(T)),$ which is a necessary condition
for $u$ being
 optimal.
\begin{remark}
On the one hand, due to (38) being a homogeneous
Malliavin-differential type equation, if $h(\cdot )$ is replaced
with $ch(\cdot )$ in the cost functional$,$ with $c$ being a
constant, we can still obtain the same result. On the other hand,
thanks to $l_2=0,$ the Hamilton function $\mathcal{H}$ is
independent of $P$, which is solution of the forward adjoint
equation in (32). Then no matter what values of $k_1,$ it does not
change the value of optimal portfolio $u.$
\end{remark}

Now we will prove that the necessary condition above is also
sufficient. For any $v_i\in \mathcal{U}$ with $i=1,2,$ we have form
the concavity of $h$ that $Eh(X^{v_1}(T))-Eh(X^{v_2}(T))\geq
E[h_x(X^{v_1}(T))(X^{v_1}(T)-X^{v_2}(T))]$, where $X^{v_i}(T)$ is
the terminal wealth corresponding to $v_i$, thus one sufficient
condition for the strategy $u$ being optimal is that
$E[h_x(X^u(T))X^v(T)]$ being a constant over $v\in \mathcal{U},$ see
Proposition 2.1 in Z. Wang \cite{WXZ}. By the above necessary
condition we have $Eh_x(X^{u}(T))X^v(T)=Eh_x(X^{u}(T))\cdot
Ee^{-A(T)}X^v(T).$ Using It\^{o} formula to
$e^{-A(t)}X^v(t)e^{\int_t^T\rho(s)ds},$ one gets
\begin{eqnarray*}
&&e^{-A(T)}X^v(T)-xe^{\int_0^T\rho(s)ds}\\
&=&\int_0^Te^{-A(s)}e^{\int_s^T\rho(s)ds}\left[v(s)\beta(s)-X(s)\frac
{\alpha(s)}{\beta(s)}\right]dW(s),
\end{eqnarray*}
hence $Ee^{-A(T)}X^v(T)=xe^{\int_0^T\rho(s)ds},$
$Eh_x(X^{u}(T))X^v(T)$ is a constant independent of $v$. Thus we
obtain
\begin{theorem}
Suppose $u$ is an optimal portfolio of the above risk minimizing
problem if and only if $u(\cdot )$ satisfies
$h_x(X^u(T))=e^{-A(T)}Eh_x(X^u(T)).$
\end{theorem}

\begin{remark}
There is one thing worthy to point out. The above argument also
implies that the sufficient condition in Proposition 2.1 in
\cite{WXZ} is also necessary. In fact, when $h$ is a concave utility
function, from our stochastic maximum principle we know that one
necessary condition for $u$ being optimal is equation (38), and this
also implies that $Ee^{-A(T)}X^v(T)$ is a constant independent of
$v.$
\end{remark}

Some special cases for the function $h(x)$ are given below to show
some exact results of $u$.

Case 1 $h(x)=x,$ we will deduce that the optimal portfolio is
$u(t)=0$. In fact, from equation (38), we have $\alpha (t)-\rho
(t)=0,$ then $\alpha (t)=\rho (t),$ which means the
stock-appreciation rate is equal to the interest rate. In this case,
the optimal portfolio is $u(t)=0.$ As we know, there are usually
many kinds of method, i.e., risk measure, to measure the wealth at
some time, for example the terminal wealth at time $T,$ and there is
one optimal portfolio for each kind of risk measure. On the other
hand, from Remark 5.2, we can choose any bounded $k_1(t,s)$, in
other words different risk measures, thus get different value
$Y(0),$ i.e., different minimal risk, while the optimal portfolio is
the same one. For example, if $k_1(t,s)=0,$ $k_2(t,s)$ is
independent of $t,$ then the minimal risk can be expressed as, $
Y(0)=e^{\int_0^T\rho
(s)+r(s)ds}x+E\int_0^Te^{\int_0^sr(u)du}k_2(s)ds. $ Note that $\rho
=0,$ it is consistent with the one in \cite{OS}.

\begin{remark}
Recently, the author \cite{W} consider the case when $r$ is allowed
to be random, while $r$ is assumed to be deterministic in \cite{Y3}.
In this case, it is easy to show that the above results also hold,
and the minimal risk is given by $Y(0)=Ee^{\int_0^T\rho
(s)+r(s)ds}x+E\int_0^Te^{\int_0^sr(u)du}k_2(s)ds.$
\end{remark}

Case 2 $h(x)=x-\frac {\gamma}{2}x^2$ with $\gamma\neq 0$ being a
constant, then $h_x(x)=1-\gamma x,$ and in the following, we denote
$F(t)=E^{\mathcal{F}_{t}}h_x(X^u(T))=1-\gamma
E^{\mathcal{F}_{t}}X(T).$ In this setting, we assume that
$E\int_0^T|v(s)|^4ds<\infty.$ Following the idea in \cite{WXZ}, we
will show the explicit form of the optimal portfolio $u$. By
equation (38), we know that
$\pi(t)=-\beta^{-1}(t)F(t)(\alpha(t)-\rho(t)).$ Using It\^{o}
formula to $A(t)F(t)$ on $[0,T]$, where $A(t)=e^{\int_0^ta(s)ds}$
and $a$ is a deterministic integral function,
\begin{eqnarray}
A(T)F(T)=F(0)+\int_0^TA(s)\pi(s)dW(s)+\int_0^TF(s)a(s)A(s)ds.
\end{eqnarray}
Since $F(T)=1-\gamma X^u(T),$ together with (39), we have
\begin{eqnarray*}
X^u(T)&=&\gamma^{-1}-(\gamma
A(T))^{-1}\left[F(0)+\int_0^TA(s)\pi(s)dW(s)+\int_0^TF(s)a(s)A(s)ds\right]\\
&=&\gamma^{-1}-\frac{F(0)}{\gamma
A(T)}-\int_0^T\frac{F(s)a(s)A(s)}{\gamma
A(T)}ds+\int_0^T\frac{A(s)F(s)(\alpha(s)-\rho(s))}{\gamma
A(T)\beta(s)}dW(s).
\end{eqnarray*}
On the other hand, by (30),
\begin{eqnarray}
X^u(T)=e^{\int_0^T\rho (s)ds}x+\int_0^T\left[e^{\int_s^T\rho
(u)du}(\alpha (s)-\rho (s))u(s)ds+e^{\int_s^T\rho (u)du}\beta
(s)u(s)dW(s)\right],
\end{eqnarray}
then by comparing the correspondent part in (39) and (40),
\[
\left\{
\begin{array}{c}
e^{\displaystyle\int_0^T\rho (s)ds}x=\gamma^{-1}-\frac{F(0)}{\gamma
A(T)}, \\
-\frac{F(s)a(s)A(s)}{\gamma A(T)}=e^{\displaystyle\int_s^T\rho
(u)du}(\alpha
(s)-\rho (s))u(s),\\
\frac{A(s)F(s)(\alpha(s)-\rho(s))}{\gamma
A(T)\beta(s)}=e^{\displaystyle\int_s^T\rho (u)du}\beta (s)u(s),
\end{array}
\right.
\]
thereby we deduce that $a(s)=\frac
{\beta^2(s)}{|\alpha(s)-\rho(s)|^2}$ , and the optimal portfolio is
expressed as
\begin{eqnarray}
 u(s)&=&e^{-\int_s^T\rho
(u)du}\frac{(\alpha(s)-\rho(s))F(s)A(s)}{\beta^2(s)\gamma
A(T)}\nonumber \\
&=&e^{-\int_s^T(\frac{|\alpha(u)-\rho(u)|^2}{\beta^2(u)}+\rho(u))du}\frac{(\alpha(s)-
\rho(s))}{\beta^2(s)}\left[\frac 1 \gamma
-E^{\mathcal{F}_{s}}X^u(T)\right].
\end{eqnarray}

To get more feeling about the general form of equation (33) in the
previous, i.e., the inhomogeneous Malliavin differential equation,
we will consider some special cases. Let $l_i(0,t)=l_i(t),$ $i=1,2$.
Moreover, for the sake of obtaining the exact expression of $P,$
$Q$, $R,$ we assume $k_1(t,s)=r(s).$

Case 1 If $\rho (t)=0,$ then for any $t\in[0,T],$ $s<t,$
\begin{eqnarray*}
P(t) &=&2r(t)e^{\int_0^tr(s)ds-\frac
12\int_0^tr^2(s)ds+\int_0^tr(s)dW(s)},
\\
Q(t) &=&-l_1(t)+l_1(t)\int_0^tP(s)ds \\
\ &=&l_1(t)\left( -1+2\int_0^tr(s)e^{\int_0^sr(u)du-\frac
12\int_0^sr^2(u)du+\int_0^sr(u)dW(u)}ds\right) , \\
R(t,s) &=&2l_1(t)\left( e^{\int_0^tr(u)du}-e^{\int_0^sr(u)du}\right)
r(s)e^{-\frac 12\int_0^sr^2(u)du+\int_0^sr(u)dW(u)},
\end{eqnarray*}
then
\begin{eqnarray*}
M(t) &=&2l_2(t)\int_0^tr(s)e^{\int_0^sr(u)du-\frac
12\int_0^sr^2(u)du+\int_0^sr(u)dW(u)}ds \\
&&+2\beta (t)E^{\mathcal{F}_t}\int_t^Tl_1(s)\left(
e^{\int_0^sr(u)du}-e^{\int_0^tr(u)du}\right) r(t)e^{-\frac
12\int_0^tr^2(u)du+\int_0^tr(u)dW(u)}ds \\
&&+\alpha (t)E^{\mathcal{F}_t}\int_t^Tl_1(s)\left(
-1+2\int_0^sr(u)e^{\int_0^ur(v)dv-\frac
12\int_0^ur^2(v)dv+\int_0^ur(v)dW(v)}du\right) ds.
\end{eqnarray*}
In fact, in this case, $l_y(s,x,y,v)=2r(s),$ and $P(\cdot )$ and
$Q(\cdot )$ satisfy
\[
\left\{
\begin{array}{c}
P(t)=2r(t)+r(t)\displaystyle\int_0^tP(s)ds+r(t)\displaystyle\int_0^tP(s)dW(s), \\
Q(t)=-l_1(t)+l_1(t)\displaystyle\int_0^tP(s)ds-\displaystyle\int_t^TR(t,s)dW(s).
\end{array}
\right.
\]
By the martingale representation theorem, there exists a unique
adapted process such that
\[
P(t)=EP(t)+\int_0^tT(t,s)dW(s),
\]
thus we have
\[
R(t,s)=l_1(t)\int_s^tT(u,s)du.
\]
Assume $r\neq 0,$ and $p^{\prime }(t)=\frac{P(t)}{r(t)},$ then we
have
\[
P^{\prime }(t)=2+\int_0^tr(s)P^{\prime }(s)ds+\int_0^tr(s)P^{\prime
}(s)dW(s),
\]
Using Ito formula to $P^{\prime \prime
}(t)=e^{-\int_0^tr(u)du}P^{\prime }(t) $ and we obtain $P^{\prime
\prime }(t)=2+\int_0^tr(s)P^{\prime \prime }(s)dW(s),$ thus we solve
that $P^{\prime \prime }(t)=2e^{-\frac
12\int_0^tr^2(s)ds+\int_0^tr(s)dW(s)},$ so we obtain $P(t)$ above.
On the other hand, $EP^{\prime \prime }(t)=2,$ then $P^{\prime
\prime }(t)=EP^{\prime \prime }(t)+\int_0^tr(s)P^{\prime \prime
}(s)dW(s)$, thus
\[
e^{-\int_0^tr(u)du}\frac{P(t)}{r(t)}=e^{-\int_0^tr(u)du}\frac{EP(t)}{r(t)}%
+\int_0^tr(s)P^{\prime \prime }(s)dW(s),
\]
thus
\[
T(t,s)=2r(t)r(s)e^{\int_0^tr(u)du-\frac
12\int_0^sr^2(u)du+\int_0^sr(u)dW(u)},
\]
thus we get
\[
R(t,s)=2l_1(t)r(s)e^{-\frac
12\int_0^sr^2(u)du+\int_0^sr(u)dW(u)}\left[
e^{\int_0^tr(v)dv}-e^{\int_0^sr(v)dv}\right] .
\]
As to the case $r=0,$ then $\forall (t,s)\in \Delta ,$ $P(t)=0,$ $%
Q(t)=-l_1(t),$ $R(t,s)=0,$ and they are all consistent with the
above results.

Case 2 If $r(s)=0,$ then we have $P(t)=0,$ and
\[
Q(t)=-l_1(t)-\rho (t)h_x(X^u(T))+\int_t^T\rho
(t)Q(s)ds-\int_t^TR(t,s)dW(s).
\]
and we can solve M-solution as follows
\begin{eqnarray*}
Q(t) &=&-l_1(t)-\rho (t)e^{\int_t^T\rho (u)du}\cdot E^{\mathcal{F}%
_t}h_x(X^u(T)) \\
&&-\rho (t)e^{\int_t^T\rho (u)du}\cdot \int_t^Te^{-\int_s^T\rho
(u)du}l_1(s)ds,
\end{eqnarray*}
and
\[
R(t,s)=-\rho (t)e^{\int_t^T\rho (s)ds}\pi (s), \quad (t,s)\in\Delta,
\]
thus
\begin{eqnarray*}
M(t) &=&\left( \beta (t)\pi (t)+(\alpha (t)-\rho (t))E^{\mathcal{F}%
_t}h_x(X^u(T))\right) \left( 1-e^{\int_t^T\rho (u)du}\right) \\
&&-(\alpha (t)-\rho (t))E^{\mathcal{F}_t}\int_t^T\left( l_1(s)+\rho
(s)\int_s^Te^{\int_s^v\rho (u)du}l_1(v)dv\right) ds.
\end{eqnarray*}

Case 3 $r(s)\neq 0,$ $\rho (s)\neq 0.$ In this case, we can obtain
the following result by combining the results in above two cases
together,
\begin{eqnarray*}
P(t) &=&2r(t)e^{\int_0^tr(s)ds-\frac
12\int_0^tr^2(s)ds+\int_0^tr(s)dW(s)},
\\
Q(t) &=&Q^{\prime }(t)-\rho (t)e^{\int_t^T\rho (u)du}\cdot E^{\mathcal{F}%
_t}h_x(X(T))+\rho (t)E^{\mathcal{F}_t}\int_t^Te^{\int_t^s\rho
(u)du}Q^{\prime }(s)ds, \\
R(t,s) &=&R^{\prime }(t,s)-\rho (t)e^{\int_t^T\rho (u)du}\pi
(s)+\rho (t)\int_t^Te^{\int_t^u\rho (v)dv}R^{\prime }(u,s)du,
\end{eqnarray*}
where
\begin{eqnarray*}
Q^{\prime }(t) &=&l_1(t)\left(
-1+2\int_0^tr(s)e^{\int_0^sr(u)du-\frac
12\int_0^sr^2(u)du+\int_0^sr(u)dW(u)}ds\right) \\
R^{\prime }(t,s) &=&2l_1(t)\left(
e^{\int_0^tr(u)du}-e^{\int_0^sr(u)du}\right) r(s)e^{-\frac
12\int_0^sr^2(u)du+\int_0^sr(u)dW(u)},
\end{eqnarray*}
thus
\begin{eqnarray*}
M(t) &=&(\alpha (t)-\rho (t))E^{\mathcal{F}_t}h_x(X^u(T))\left(1-
e^{\int_t^T\rho (u)du}\right) +\beta (t)\pi (t)\left(1-
e^{\int_t^T\rho
(u)du}\right) \\
&&+(\alpha (t)-\rho (t))E^{\mathcal{F}_t}\int_t^T\left( Q^{\prime
}(s)+\rho
(s)\int_s^Te^{\int_s^v\rho (u)du}Q^{\prime }(v)dv\right) ds \\
&&+l_2(t)\left(\int_0^t2r(s)e^{\int_0^sr(u)du-\frac
12\int_0^sr^2(u)du+\int_0^sr(u)dW(u)}ds\right)\\
&&+\beta (t)E^{\mathcal{F}%
_t}\int_t^T\left( R^{\prime }(s,t)+\rho (s)\int_s^Te^{\int_s^v\rho
(u)du}R^{\prime }(v,t)dv\right) ds.
\end{eqnarray*}

\end{document}